\documentclass[10pt,reqno]{amsart}

\usepackage{amssymb}
\usepackage{amscd}
\usepackage{amsfonts}
\usepackage{amsmath}
\usepackage{setspace}
\usepackage{version}

\newtheorem{maintheorem}{Theorem}
\newtheorem{theorem}{Theorem}[section]
\newtheorem{lemma}[theorem]{Lemma}

\theoremstyle{definition}

\theoremstyle{remark}

\newsavebox{\proofbox}
\savebox{\proofbox}{\begin{picture}(7,7)  \put(0,0){\framebox(7,7){}}\end{picture}}

\newcommand{\md}[1]{\ensuremath{\,(\operatorname{mod}\, #1)}}

\renewcommand{\leq}{\leqslant}
\renewcommand{\geq}{\geqslant}

\def\R{\mathbf{R}}

\def\Z{\mathbf{Z}}
\def\E{\mathbf{E}}
\def\P{\mathbf{P}}

\def\N{\mathbf{N}}

\def\eps{\varepsilon}

\numberwithin{equation}{section}

\begin{document}

\title{F{\o}lner sequences and sum-free sets}



\author{Sean Eberhard}
\address{Sean Eberhard, Centre for Mathematical Sciences, University of Cambridge}
\email{s.eberhard@dpmms.cam.ac.uk}

\onehalfspace

\begin{abstract}
Erd\H{o}s showed that every set of $n$ positive integers contains a subset of size at least $n/(k+1)$ containing no solutions to $x_1 + \cdots + x_k = y$. We prove that the constant $1/(k+1)$ here is best possible by showing that if $(F_m)$ is a multiplicative F{\o}lner sequence in $\N$ then $F_m$ has no $k$-sum-free subset of size greater than $(1/(k+1)+o(1))|F_m|$.
\end{abstract}
\maketitle


\section{Introduction}
Let $k\geq 2$ be an integer. A subset $A$ of an abelian group is called $k$-sum-free if there do not exist $a_1, \dots, a_k \in A$ such that $a_1 + \cdots + a_k \in A$. In 1965 Erd\H{o}s~\cite{erdos1} proved with the following ingenious argument that every set $A$ of $n$ positive integers has a $k$-sum-free subset of size at least $n/(k+1)$. Since the open interval $S \subset \R/\Z$ of length $1/(k+1)$ centred at $1/(2k-2)$ is $k$-sum-free, it follows that for each $x\in \R/\Z$ the set $A_x$ of $a\in A$ such that $ax\in S$ is also $k$-sum-free. But if $x$ is chosen uniformly at random from $\R/\Z$ then for each $a\in A$ the product $ax$ also has the uniform distribution, so by linearity of expectation the expected size of $A_x$ is $n/(k+1)$. In particular $|A_x|\geq n/(k+1)$ for some $x$.

Our main theorem is that the constant $1/(k+1)$ in this theorem cannot be improved: for every $\eps>0$ there is a set of $n$ positive integers containing no $k$-sum-free subset of size greater than $(1/(k+1)+\eps)n$. In fact, we can give explicit examples of sets with no large $k$-sum-free subsets. Call a sequence $(F_m)$ of subsets $F_m\subset\N$ a \emph{F{\o}lner sequence in $(\N,\cdot)$} if, for every fixed $a\in\N$,
\[
	\frac{|(a\cdot F_m)\triangle F_m|}{|F_m|} \longrightarrow 0\quad\text{as }m\to\infty,
\]
where $\triangle$ denotes symmetric difference. For example, if
\[
	F_m = \{p_1^{e_1}\cdots p_m^{e_m} : 0\leq e_i<m\},
\]
where $p_1,p_2,\dots$ are the primes, then $(F_m)$ is a F{\o}lner sequence in $(\N,\cdot)$. For any such sequence, the sets $F_m$ eventually have no large $k$-sum-free subsets.

\begin{maintheorem}\label{thm:main}
If $(F_m)$ is a F{\o}lner sequence in $(\N,\cdot)$ then $F_m$ has no $k$-sum-free subset of size greater than $(1/(k+1) + o(1))|F_m|$.
\end{maintheorem}

We begin by giving a quick deduction of the case $k=2$ of this theorem from previous work. We then give two proofs in the general case: one short, infinitary, and ineffective, the other longer, finitary, and effective, though with poor bounds. In both proofs we rely on the theorem of {\L}uczak and Schoen~\cite{luczakschoen} that every maximal $k$-sum-free subset of $\N$ of upper density greater than $1/(k+1)$ is periodic. In the first proof we use this theorem as a black box, while for the second we prove a finitary version by closely following~\cite{luczakschoen}.

Ben Green, Freddie Manners, and I~\cite{EGM} recently proved that there is a set of $n$ positive integers containing no $2$-sum-free subset of size larger than $(1/3+o(1))n$. The method we used extends to the case of $3$-sum-free sets with a little work, but the method does not seem to extend easily to $k$-sum-free sets for $k>3$. Until now, the best result known in the case $k>3$ was the theorem of Bourgain~\cite{bourgain} that if $\delta_k$ is the largest constant such that every set of $n$ positive integers contains a $k$-sum-free set of size at least $\delta_k n$ then $\delta_k\to 0$ as $k\to\infty$.

\section{Deduction of the case $k=2$ from previous work}\label{sec:deduction}

Fix a F{\o}lner sequence $(F_m)$ in $(\N,\cdot)$. The F{\o}lner property will be used in the following way: if $x\in F_m$ is uniformly random and $E(x)$ is some event depending on $x$, then for every fixed $a\in\N$ we have
\[
	|\P(E(ax))-\P(E(x))| \leq \frac{|(a\cdot F_m)\triangle F_m|}{|F_m|}\longrightarrow 0 \quad\text{as }m\to\infty.
\]
This allows us to imitate Erd\H{o}s's argument with $(F_m)$ in place of $\R/\Z$, as in the following lemma.

\begin{lemma}\label{lem:erdosfolner}
The following two statements are equivalent.
\begin{enumerate}
	\item For infinitely many $m$, $F_m$ has a $k$-sum-free subset of size at least $\delta |F_m|$.
	\item For every finite $A\subset\N$, $A$ has a $k$-sum-free subset of size at least $\delta|A|$.
\end{enumerate}
\end{lemma}
\begin{proof}
Of course (2) trivially implies (1), so it suffices to prove (1) implies (2). Suppose $F_m$ has a $k$-sum-free subset $S_m$ of size at least $\delta|F_m|$. For $x\in F_m$, let $A_x$ be the set of all $a\in A$ such that $ax\in S_m$. Then $A_x$ is $k$-sum-free, and if we choose $x\in F_m$ uniformly at random then the expected size of $A_x$ is
\begin{align*}
	\E(|A_x|) 
		&= \sum_{a\in A} \P(a\in A_x) = \sum_{a\in A} \P(ax\in S_m)\\
		&\geq \sum_{a\in A} \left(\P(x\in S_m) - \frac{|(a\cdot F_m)\triangle F_m|}{|F_m|}\right)\\
		&\geq \delta|A| - \sum_{a\in A} \frac{|(a\cdot F_m)\triangle F_m|}{|F_m|}.
\end{align*}
Hence from the F{\o}lner property and the integrality of $|A_x|$ it follows that for sufficiently large $m$ there is some $x\in F_m$ such that $|A_x|\geq\delta|A|$.
\end{proof}

By~\cite{EGM}, for every $\eps>0$ there is a set $A$ of $n$ positive integers with no $2$-sum-free subset of size greater than $(1/3+\eps)n$. From this and the above lemma we deduce the case $k=2$ of Theorem~\ref{thm:main}.

\section{Infinitary proof of Theorem~\ref{thm:main}}\label{sec:ineffproof}

In this section we assume basic familiarity with ultrafilters, and in particular Loeb measure. The reader needing an introduction might refer to Bergelson and Tao~\cite[Section 2]{bergelsontao}.

Again fix a F{\o}lner sequence $(F_m)$ in $(\N,\cdot)$, and assume for infinitely many $m$ that $F_m$ has a $k$-sum-free subset $S_m$ of size at least $\delta|F_m|$. By passing to a subsequence we may assume this holds for all $m$.

\begin{lemma}
There is an abelian group $X$, a $\sigma$-algebra $\Sigma$ of subsets of $X$, a probability measure $\mu$ on $\Sigma$, and a set $S\in\Sigma$ such that \textup{(1)} for every $a\in\N$ the map $x\mapsto ax$ is $\Sigma$-measurable and $\mu$-preserving, and \textup{(2)} $S$ is $k$-sum-free and $\mu(S)\geq\delta$.
\end{lemma}
\begin{proof}
Let $p\in\beta\N\setminus\N$ be a nonprincipal ultrafilter, let $X$ be the ultraproduct $\prod_{m\to p}\Z$, and let $\Sigma$ be the Loeb $\sigma$-algebra on $X$. Defining $\mu_m$ on subsets of $\Z$ by
\[
	\mu_m(A) = |A\cap F_m|/|F_m|,
\]
let $\mu$ be the Loeb measure induced by $(\mu_m)$. Let $S$ be the internal set $\prod_{m\to p} S_m$.

To verify \textup{(1)}, note that $x\mapsto ax$ sends internal sets to internal sets, so it is measurable. Moreover $x\mapsto ax$ approximately preserves $\mu_m$ by the F{\o}lner property, so it exactly preserves the Loeb measure $\mu$. For \textup{(2)}, it follows from the basic properties of ultrafilters that $S$ is $k$-sum-free, and by definition of $\mu$ we have
\[
	\mu(S) = \text{st}\left(\lim_{m\to p} \mu_m(S_m)\right) \geq \delta. \qedhere
\]
\end{proof}

We define the \emph{upper density} of a set $A\subset\N$ by
\[
	\overline{d}(A) = \limsup_{n\to\infty} \frac{|A\cap\{1,\dots,n\}|}{n},
\]
and we define its \emph{upper density on multiples} by
\[
	\widetilde{d}(A) = \limsup_{N\to\infty} \overline{d}(A/N!),
\]
where $A/N! = \{a\in\N : N!a\in A\}$. Equivalently,
\[
	\widetilde{d}(A) = \limsup_{N\to\infty}\limsup_{n\to\infty} \frac{|A\cap\{N!, 2N!, \dots, nN!\}|}{n}.
\]

\begin{lemma}
There is a $k$-sum-free subset $A$ of $\N$ such that $\widetilde{d}(A)\geq\delta$.
\end{lemma}
\begin{proof}
Let $X$, $\Sigma$, $\mu$, and $S$ be as in the previous lemma. For $x\in X$, let $A_x$ be the set of all $a\in\N$ such that $ax\in S$. Then $A_x$ is $k$-sum-free, $\widetilde{d}(A_x)$ is a $\Sigma$-measurable function of $x$, and if $x$ is chosen randomly with law $\mu$ then, by Fatou's lemma,
\begin{align*}
	\E(\widetilde{d}(A_x))
		&\geq \limsup_{N\to\infty} \limsup_{n\to\infty} \E\left(\frac{|A_x\cap\{N!,2N!,\dots,nN!\}|}{n}\right)\\
		&= \limsup_{N\to\infty}\limsup_{n\to\infty} \frac{1}{n}\sum_{a=1}^n \P(aN!\in A_x)\\
		&= \limsup_{N\to\infty}\limsup_{n\to\infty} \frac{1}{n}\sum_{a=1}^n \P(aN!x\in S)\\
		&= \limsup_{N\to\infty}\limsup_{n\to\infty} \frac{1}{n}\sum_{a=1}^n \P(x\in S)\\
		&= \mu(S)\geq \delta.\qedhere
\end{align*}
\end{proof}

To complete the proof of Theorem~\ref{thm:main}, it suffices to prove that there is no $k$-sum-free subset of $\N$ of upper density on multiples larger than $1/(k+1)$. We use the following theorem of {\L}uczak and Schoen.

\begin{theorem}[\cite{luczakschoen}]
Every $k$-sum-free subset of $\N$ of upper density larger than $1/(k+1)$ is contained in a periodic $k$-sum-free set.
\end{theorem}

\begin{lemma}
Every $k$-sum-free $A\subset\N$ satisfies $\widetilde{d}(A)\leq 1/(k+1)$.
\end{lemma}
\begin{proof}
If $\widetilde{d}(A)>0$ then for every $N$ the set $A/N!$ contains a multiple of every natural number. In particular $A/N!$ is not contained in a periodic $k$-sum-free set, so, by the {\L}uczak-Schoen theorem,
\[
	\widetilde{d}(A) = \limsup_{N\to\infty} \overline{d}(A/N!) \leq 1/(k+1).\qedhere
\]
\end{proof}

\section{A finitary {\L}uczak-Schoen theorem}\label{sec:finitaryluczakschoen}

For a set $A\subset\N$, let $A_x = \{a\in A: a\leq x\}$, let $d_n(A) = |A_n|/n$, and let
\[
	A_n - (k-1)A_n = \{u - v_1 - \cdots - v_{k-1} : u,v_1,\dots,v_{k-1}\in A_n\}.
\]

\begin{lemma}
Suppose that $A\subset\N$ is $k$-sum-free, that $A_{n_0}$ contains an arithmetic progression $x, x+m, \dots, x+(i-1)m$, and that some $d\in A_{n_0}-(k-1)A_{n_0}$ satisfies $d\equiv x\md{m}$. Then for every $\eps>0$ and every sequence $n_1,n_2,\dots$ of naturals such that $n_{j+1}\geq\eps^{-1}n_j$ for each $j\geq 0$ there is some $\ell\leq kn_0$ such that
\[
	d_{n_\ell}(A) \leq \frac{i+k-2}{i(k+1)+k-3} + 4k\eps.
\]
\end{lemma}
\begin{proof}
Assume first that $d<x$. Let
\[
	B=\{a\in A: a+jm\in A \text{ for some } j\in\{1,\dots,i\}\}.
\]
Then the sets $A, A+(k-1)x, A\setminus B+(k-1)x+m, \dots, A\setminus B+(k-1)x+(i-1)m$ are pairwise disjoint, so
\begin{equation}\label{|B_n|}
	(i+1)|A_n|-(i-1)|B_n| \leq n + (k-1)x + (i-1)m.\tag{$\ast$}
\end{equation}

For each $t\in\Z$ let $u(t)$ be the smallest element of $A$ such that
\[
	t = u(t) - v_1 - \cdots - v_{k-1}
\]
for some $v_1,\dots,v_{k-1}\in A$; if no such element exists let $u(t)=\infty$. Since $u(d)\leq n_0$ and $u(x+jm)=\infty$ for each $j\in\{0,\dots,i-1\}$, and since $d\geq -(k-1)n_0$, we may find $d'$ and $\ell$ such that $1\leq\ell\leq kn_0$, $u(d')\leq \eps n_\ell$, and $u(d'+jm)>n_\ell$ for each $j\in\{1,\dots,i\}$.

Write $d' = u - v_1 -\cdots - v_{k-1}$, where $u, v_1, \dots, v_{k-1}\in A_{\eps n_\ell}$. Then the sets $A, B+v_1+\cdots+v_{k-1}, B+u+v_1+\cdots+v_{k-2}, \dots, B+(k-2)u + v_1, A+(k-1)u$ are pairwise disjoint when restricted to $\{1,\dots,n_\ell-im\}$. Indeed, being $k$-sum-free certainly $A$ is disjoint from the others, while if for some $0\leq s<t\leq k-1$ we have
\[
	(B + su + v_1 + \cdots + v_{k-s-1})\cap(A + tu + v_1 + \cdots + v_{k-t-1})\cap\{1,\dots,n_\ell-im\} \neq\emptyset
\]
then there exists $b\in B_{n_\ell-im}$ and $a\in A$ such that
\[
	b+su + v_1 + \cdots + v_{k-s-1} = a + tu + v_1 + \cdots + v_{k-t-1},
\]
whence
\[
	d' = (b-a) - (t-s-1)u - v_1 - \cdots - v_{k-t-1}-v_{k-s}-\cdots -v_{k-1}.
\]
But since $b\in B$ there is some $j\in\{1,\dots,i\}$ such that $b+jm\in A$, and so the above equation shows that $b+jm\geq u(d'+jm) > n_\ell$, contradicting $b\leq n_\ell - im$. Thus
\[
	2|A_{n_\ell}| + (k-1)|B_{n_\ell}|\leq n_\ell + (k+1)im + k(k-1)u,
\]
and the lemma is proved by combining this inequality with \eqref{|B_n|}.

The case $d>x$ is similar, but we consider
\[
	B=\{a\in A: a-jm\in A \text{ for some }j\in\{1,\dots,i\}\},
\]
and we find $d'$ and $\ell$ such that $1\leq\ell\leq kn_0$, $u(d')\leq \eps n_\ell$, and $u(d'-jm)>n_\ell$ for each $j\in\{1,\dots,i\}$.
\end{proof}

%
%

\begin{theorem}\label{thm:fLS}
For every $k\geq2$ and $\eps>0$ there exist natural numbers $N=N(k,\eps)$ and $Q=Q(k,\eps)$ such that if $n_0\geq N$ and $A$ is $k$-sum-free such that
\[
	d_{n_0}(A)\geq\frac{1}{k+1}+\eps,
\]
then either $A_{n_0}$ is contained in a $Q$-periodic $k$-sum-free set, or for every sequence $n_1,n_2,\dots$ of naturals such that $n_{j+1}\geq 16k\eps^{-1} n_j$ for each $j\geq 0$ there is some $\ell\leq kn_0$ such that
\[
	d_{n_\ell}(A) \leq \frac{1}{k+1} + \eps/2.
\]
\end{theorem}

\begin{proof}
Choose $i=i(k,\eps)$ so that
\[
	\frac{i+k-2}{i(k+1)+k-3} \leq \frac{1}{k+1} + \eps/4,
\]
and then choose $N=N(k,\eps)$ and $Q=Q(k,\eps)$ so that if $n_0\geq N$ every subset of $\{1,\dots,n_0\}$ of size at least $(\eps/2) n_0$ contains an arithmetic progression of length $i$ and common difference dividing $Q$, and so that
\[
	\frac{(k-1)k}{k+1}\frac{Q}{N} < \eps/2.
\]
The existence of $N$ and $Q$ follows from Szemer{\'e}di's theorem.

Suppose $d_{n_0}(A)\geq1/(k+1)+\eps$. Let
\[
	R = \{r\in\N : r\equiv a\md{Q}\text{ for some }a\in A_{n_0}\},
\]
and
\[
	D = \{r\in R: r + t_1 + \cdots + t_{k-1}\notin R \text{ for each }t_1,\dots,t_{k-1}\in R\}.
\]
If $A_{n_0}$ is not contained in a $Q$-periodic $k$-sum-free set then $R$ must not be $k$-sum-free. Suppose $x_1,\dots,x_k,x\in R$ and $x_1+\cdots+x_k=x$, where we may assume $x\leq kQ$. Since $D$ is $k$-sum-free the sets $D,D+x_1+\cdots+x_{k-1},D+x+x_1+\cdots+x_{k-2},\dots,D+(k-1)x$ are disjoint, so
\[
	(k+1)|D_{n_0}|\leq n_0 + (k-1)kQ,
\]
whence
\[
	d_{n_0}(A\setminus D) \geq \eps - \frac{(k-1)k}{k+1}\frac{Q}{n_0} \geq \eps/2.
\]
It follows that $A_{n_0}\setminus D_{n_0}$ contains an arithmetic progression $x, x + m, \dots, x + (i-1)m$, where $m$ divides $Q$. By definition of $R$ and $D$ there exists $u,v_1,\dots,v_{k-1}\in A_{n_0}$ such that $x + v_1 + \cdots + v_{k-1} = u\md{Q}$. But then $d = u-v_1-\cdots-v_{k-1}\equiv x\md{m}$, so by the lemma there is some $\ell\leq kn_0$ such that
\[
	d_{n_\ell}(A) \leq \frac{i+k-2}{i(k+1)+k-3} + \eps/4 \leq \frac{1}{k+1} + \eps/2.\qedhere
\]
\end{proof}

We may easily recover the original {\L}uczak-Schoen theorem from the above finitary version. Indeed, for $k\geq 2$ and $\eps>0$, let $N$ and $Q$ be as in the above theorem, and suppose $A$ is $k$-sum-free and $d_{n_i}(A)\geq1/(k+1)+\eps$ for each $i$, where $n_i\to\infty$. By passing to a subsequence of $(n_i)$ we may assume $n_0\geq N$, $n_{j+1}\geq 16k\eps^{-1}n_j$ for each $j\geq 0$, and, if $A$ is not contained in a $Q$-periodic $k$-sum-free set, $A_{n_0}$ is not contained in a $Q$-periodic $k$-sum-free set. But then by the above theorem there is some $\ell$ such that $d_{n_\ell}(A)\leq 1/(k+1)+\eps/2$, a contradiction.

\section{Finitary proof of Theorem~\ref{thm:main}}\label{sec:effproof}

\begin{lemma}
Fix $\eps>0$ and let $N$ and $Q$ be as in Theorem~\ref{thm:fLS}. Then for every $n_0\geq N$ there exists a finitely supported measure $\nu$ on $\N$ such that if $A$ is $k$-sum-free and $A_{n_0}$ is not contained in a $Q$-periodic $k$-sum-free set then $\nu(A)\leq 1/(k+1)+2\eps$.
\end{lemma}
\begin{proof}
Continue $n_0$ to a sequence $(n_i)$ satisfying $n_{i+1}\geq 16k\eps^{-1}n_i$ for each $i\geq 0$, and let $(i_s)$ be a sequence of indices such that $i_{-1}=-1$, $i_0=0$, and $i_{s+1}-i_s\geq 2\eps^{-1} kn_{i_s}$ for each $s\geq 0$. Define measures $\nu_s$ by
\[
	\nu_s(A) = \frac{1}{i_s-i_{s-1}}\sum_{i=i_{s-1}+1}^{i_s} d_{n_i}(A),
\]
and define $\nu$ by
\[
	\nu(A) = \frac{1}{t+1}\sum_{s=0}^t \nu_s(A),
\]
where $t\geq 2\eps^{-1}$.

Suppose that $A$ is $k$-sum-free, and let $s_0$ be the least $s$ such that $\nu_s(A)\geq 1/(k+1)+\eps$: if no such $s$ exists then $\nu(A)\leq 1/(k+1)+\eps$, so we are done. Find $i_0$ such that $i_{s_0-1}<i_0\leq i_{s_0}$ and $d_{n_{i_0}}(A)\geq 1/(k+1)+\eps$. If $A_{n_{i_0}}$ is not contained in a $Q$-periodic $k$-sum-free set then by Theorem~\ref{thm:fLS} there are at most $kn_{i_0}$ indices $i>i_{s_0}$ such that $d_{n_i}(A)\geq 1/(k+1)+\eps/2$. Since $kn_{i_0}\leq kn_{i_{s_0}} \leq (\eps/2)(i_{s_0+1} - i_{s_0})$ we find that
\[
	\nu(A) \leq 1/(k+1)+\eps + 1/(t+1) + \eps/2 \leq 1/(k+1)+2\eps.\qedhere
\]
\end{proof}

The next lemma uses an idea based on the contraction mapping theorem which also appeared in \cite{EGM}.

\begin{lemma}
For every $\eps>0$ there is a finitely supported measure $\mu$ on $\N$ such that every $k$-sum-free set $A$ satisfies $\mu(A)\leq 1/(k+1) + 4\eps$.
\end{lemma}
\begin{proof}
Let $\nu_{n_0}$ be the measure constructed by the previous lemma for $n_0\geq N$, and let $\tau:\N\to\N$ be the map $x\mapsto Qx$. Define measures $\mu_i$ inductively as follows. Start with $\mu_1=\nu_N$, and thereafter if $\mu_i$ is supported on $\{1,\dots,M_i\}$ let
\[
	\mu_{i+1} = \frac{k}{k+1}\tau_\ast\mu_i + \frac{1}{k+1}\nu_{QM_i}.
\]

If $A$ is $k$-sum-free and $\nu_{QM_i}(A)>1/(k+1)+2\eps$ then by the previous lemma $A_{QM_i}$ is contained in a $Q$-periodic $k$-sum-free set. In particular $A_{QM_i}$ is disjoint from $Q\N$, so $\tau_\ast\mu_i(A)=0$. Hence in this case $\mu_{i+1}(A)\leq 1/(k+1)$.

If on the other hand $\nu_{QM_i}(A)\leq 1/(k+1)+2\eps$ then by induction we have
\begin{align*}
	\mu_{i+1}(A) &\leq \frac{k}{k+1}\left(\frac{1}{k+1}+2\eps + \left(\frac{k}{k+1}\right)^i\right) + \frac{1}{k+1}\left(\frac{1}{k+1}+2\eps\right)\\
	&= \frac{1}{k+1}+2\eps + \left(\frac{k}{k+1}\right)^{i+1}.
\end{align*}

Hence if $i$ is chosen large enough that $(k/(k+1))^i\leq 2\eps$ then $\mu_i(A)\leq 1/(k+1)+4\eps$ for every $k$-sum-free $A\subset\N$.
\end{proof}


To finish the proof of Theorem~\ref{thm:main} we need a version of Lemma~\ref{lem:erdosfolner} for measures.

\begin{lemma}
Suppose that $F$ satisfies
\[
	\frac{|(a\cdot F)\triangle F|}{|F|} \leq \eps
\]
for every $a\in\{1,\dots,n\}$, and that $F$ has a $k$-sum-free subset $S$ of size at least $\delta|F|$. Then for every probability measure $\mu$ supported on $\{1,\dots,n\}$ there is a $k$-sum-free set $A$ for which $\mu(A) \geq \delta - \eps$.
\end{lemma}
\begin{proof}
For $x\in F$, let $A_x$ be the set of all $a\in\{1,\dots,n\}$ such that $ax\in S$. Then $A_x$ is $k$-sum-free, and if we choose $x\in F$ uniformly at random then the expected measure of $A_x$ is
\begin{align*}
	\E(\mu(A_x)) 
		&= \int \P(a\in A_x)\,d\mu(a) = \int \P(ax\in S)\,d\mu(a)\\
		&\geq \int \left(\P(x\in S) - \frac{|(a\cdot F)\triangle F|}{|F|}\right)\,d\mu(a)\\
		&\geq \delta-\eps.\qedhere
\end{align*}
\end{proof}

Theorem~\ref{thm:main} follows from the previous two lemmas.

\section{Final remarks}

{\L}uczak and Schoen~\cite{luczakschoen} also considered so-called strongly $k$-sum-free sets, sets which are $\ell$-sum-free for each $\ell=2,\dots,k$. They prove that every maximal strongly $k$-sum-free subset of $\N$ of upper density larger than $1/(2k-1)$ is periodic. Using this theorem one may easily modify Section \ref{sec:ineffproof} to verify that for any F{\o}lner sequence $(F_m)$ in $(\N,\cdot)$ the sets $F_m$ contain no strongly $k$-sum-free subset of size larger than $(1/(2k-1) + o(1))|F_m|$. One may also adapt the methods of Sections \ref{sec:finitaryluczakschoen} and \ref{sec:effproof} to give an effective proof of this theorem. We leave the details to the energetic reader.

\subsubsection*{Acknowledgements} I am grateful to Ben Green and to Freddie Manners for helpful comments and discussion.

\bibliography{star-free}{}
\bibliographystyle{alpha}
\end{document}